\documentclass[reqno,11pt]{amsart}

\usepackage{amsthm, amsmath, amssymb}
\usepackage[utf8]{inputenc}
\usepackage[T1]{fontenc}

\usepackage[hidelinks]{hyperref}

\usepackage{bm}
\usepackage[margin=1in]{geometry}

\usepackage[noabbrev,capitalize,nameinlink]{cleveref}
\crefname{equation}{}{}

\newtheorem{theorem}{Theorem}[section]

\newtheorem{lemma}[theorem]{Lemma}

\newtheorem{conjecture}[theorem]{Conjecture}

\theoremstyle{definition}

\newtheorem{example}[theorem]{Example}

\theoremstyle{remark}
\newtheorem*{remark}{Remark}

\newcommand{\abs}[1]{\left\lvert#1\right\rvert}

\newcommand{\ceil}[1]{\left\lceil #1 \right\rceil}
\newcommand{\paren}[1]{\left( #1 \right)}

\newcommand{\CC}{\mathbb{C}}

\newcommand{\FF}{\mathbb{F}}
\newcommand{\RR}{\mathbb{R}}

\newcommand{\ZZ}{\mathbb{Z}}

\newcommand{\cL}{\mathcal L}
\newcommand{\cJ}{\mathcal J}

\newcommand{\cF}{\mathcal F}

\title{Joints tightened}

\author{Hung-Hsun Hans Yu}
\author{Yufei Zhao}

\address{Yu:   Department of Mathematics, Princeton University, Princeton, NJ }
\email{hansonyu@princeton.edu}

\address{Zhao: Department of Mathematics, Massachusetts Institute of Technology, Cambridge, MA, USA}
\email{yufeiz@mit.edu}

\begin{document}

\begin{abstract}
	Given a set of lines in $\mathbb{R}^3$, a \emph{joint} is a point passed through by three lines not all lying in some plane.	The joints problem asks for the maximum number of joints formed by $L$ lines.
	Using the polynomial method, Guth and Katz proved that the answer is $O(L^{3/2})$, which is best possible up to a constant factor.
	We prove a new upper bound of $(\sqrt{2}/3) L^{3/2}$, which matches, up to a $1+o(1)$ factor, the best known construction: place $k$ generic planes, and use their pairwise intersections to form $\binom{k}{2}$ lines and their triple-wise intersections to form $\binom{k}{3}$ joints. 
	Guth conjectured that this construction is optimal. 
	Our result generalizes to arbitrary dimensions and fields.
	
	Our technique builds on the work on Ruixiang Zhang proving the multijoints conjecture via an extension of the polynomial method. We set up a variational problem to control the high order of vanishing of a polynomial at each joint. 
\end{abstract}

\maketitle

\section{Introduction}

The joints problem asks to determine, given a set of $L$ lines in $\RR^d$ (or $\FF^d$, where throughout this paper $\FF$ denotes an arbitrary field), what is the maximum number of joints formed by these lines. Given a set of lines, a \emph{joint} is defined to be a point passed through by $d$ lines from the set that do not all lie in the same hyperplane. Note that we only count each point at most once as a joint, even if more than $d$ lines pass through it. Here one should think of $d$ as fixed and $L$ as large. The following example appears in the original paper on the joints problem~\cite{CEGPSSS92}.

\begin{example} \label{ex:tight}
Take $k$ generic hyperplanes (provided that $|\FF|$ is large enough) forming $L = \binom{k}{d-1}$ lines from $(d-1)$-wise intersections and
\[
J = \binom{k}{d} = \paren{\frac{(d-1)!^{1/(d-1)}}{d} - o(1)} L^{d/(d-1)}
\]
joints from $d$-wise intersections. Here $d$ is fixed and $L \to \infty$.
\end{example}

In their breakthrough work, Guth and Katz~\cite{GK10} proved that, in $\RR^3$, \cref{ex:tight} has the maximum number of joints up to a constant factor.
Their work was a precursor to their celebrated solution to the Erd\H{o}s distinct distances problem~\cite{GK15}.
Their proof used the polynomial method~\cite{Guth16}, as inspired by Dvir's stunning solution to the finite field Kakeya problem~\cite{Dvir09}. 
The joints theorem was extended to general dimensions independently by Quilodr\'an~\cite{Qui09} and Kaplan--Sharir--Shustin~\cite{KSS10}, and these techniques and results also extend to arbitrary fields (also see~\cite{CI14,Dvir10,Tao14}).

\begin{theorem}[Joints theorem~\cite{GK10,KSS10,Qui09}] \label{thm:joints}
	For every $d$ there is some constant $C_d$ so that $L$ lines in $\FF^d$ form at most $C_d L^{d/(d-1)}$ joints.
\end{theorem}

Guth~\cite[Section 2.5]{Guth16} conjectured (at least for $d=3$) that \cref{ex:tight} is the exact optimum, i.e., $\binom{k}{d-1}$ lines in $\FF^d$ form at most $\binom{k}{d}$ joints for every positive integer $k$. Our main result is a new upper bound on the number of joints confirming Guth's conjecture up to a $1+o(1)$ factor.

\begin{theorem}[Main theorem] \label{thm:main}
	The number of joints formed by $L$ lines in $\FF^d$ is at most
	\[
	\frac{(d-1)!^{1/(d-1)}}{d} L^{d/(d-1)}.
	\]
\end{theorem}

Our result is a rare instance in incidence geometry where the sharp constant is determined. 
Very recently (after this paper first appeared), the finite field Kakeya problem was solved up to a $1+o(1)$ factor \cite{BW21,DKSS13,SS08}.
The only other sharp result that we are aware of is Green and Tao's theorem on ordinary lines~\cite{GT13}.
In contrast, in many classical results, such as the Szemer\'edi--Trotter theorem~\cite{ST83}, the exact constant factor is unknown.

The basic idea of the polynomial method in incidence geometry is that one can assert the existence of a nonzero polynomial that satisfies various linear constraints on its coefficients provided that there are enough degrees of freedom. 
A typical constraint asks the polynomial to vanish at some point.

Zhang~\cite{Zhang} proved a generalization of the joints problem, known as the multijoints problem (see \cref{sec:multijoint}), by studying higher order vanishings of a polynomial. 
The method of higher order vanishings was also used earlier to improve the upper bound to the finite field Kakeya problem~\cite{DKSS13}. 

Our main innovation is to set up a variational problem whose variables track the desired orders of vanishing of a polynomial at each joint in the configuration.
In contrast to earlier approaches, we do not specify in advance the order of vanishing, and instead show via a compactness argument that the associated variational problem has a desirable optimum. 
Curiously, our argument is implicit in two ways: the existence of the polynomial as well as the existence of a good choice of vanishing orders.

Wolff~\cite{Wol99} observed a connection between the joints problem and the Kakeya problem, a central problem in harmonic analysis. The joints problem can be viewed as a discrete analog of the multilinear Kakeya problem, which was solved in the non-endpoint case by Bennett, Carbery, and Tao~\cite{BCT06}, and then later solved in the endpoint case by Guth \cite{Guth10} using the polynomial method. Roughly speaking, in the Kakeya setup, one considers families of well separated tubes and the goal is to upper bound the measure of the set of points that are contained in many tubes.

\cref{sec:proof} contains the complete proof of \cref{thm:main}.
In the rest of the paper, we extend our arguments to some variants of the joints problem.
In \cref{sec:multijoint}, we apply the technique to improve the upper bound in the multijoints problem.
In \cref{sec:higher-dimension}, we consider a higher dimensional generalization of the joints problem where lines are replaced by flats.

For all these variants of the joints problem, we conjecture that the optimal configurations also come from extensions of \cref{ex:tight} (see \cref{sec:multijoint,sec:higher-dimension}). These special geometric configurations can be encoded by edge-colored hypergraphs, so that the problem becomes a multicolored Kruskal--Katona-type extremal set theory/hypergraph problem. Curiously, our proof of \cref{thm:main} gives a seemingly new polynomial method proof of an old fact that cliques asymptotically maximize the triangle density in a graph of given edge density.

\section{Proof of the main theorem} \label{sec:proof}

A \emph{joints configuration} $(\cJ, \cL)$ consists of a set $\cL$ of lines in $\FF^d$ and  a set $\cJ$ of joints formed by these lines. 
We abuse the notation ``$p\in\ell$'' slightly to deal with the cases when some point has more than $d$ lines passing through it. At each joint $p$, we pick arbitrarily $d$ lines from $\cL$ passing through $p$ in linearly independent directions, and we write ``$p \in \ell$'' (and likewise with ``$\ell$ contains $p$'', etc.) only if $\ell$ is one of these $d$ chosen lines. 

Given a polynomial $g \in \FF[x_1, \dots, x_d]$, a joint $p \in \cJ$, an ordering $\ell_1, \dots, \ell_d$ of the $d$ lines passing through $p$, and a vector $(\beta_1, \dots, \beta_d)$ of nonnegative integers, we say that $g$ \emph{vanishes to order $(\beta_1, \dots, \beta_d)$ at $p$ in directions $(\ell_1, \dots, \ell_d)$} if the coefficient of $x_1^{\omega_1} \cdots x_d^{\omega_d}$ in $g(p + x_1 e_{\ell_1} + \cdots + x_d e_{\ell_d})$ is zero whenever $0 \le \omega_i < \beta_i$ for each $i \in [d]$, where $e_{\ell_i}$ is any nonzero vector parallel to $\ell_i$. (In particular, the condition holds vacuously if $\beta_i=0$ for any $i$.)

Furthermore, given a nonnegative integer $\beta_{p,\ell}$ for each pair $(p,\ell) \in \cJ \times \cL$ with $p \in \ell$, we say that the polynomial $g$ \emph{vanishes to order $\bm \beta = (\beta_{p,\ell})_{(p,\ell) : p \in \ell} \in \ZZ^{|\cJ|d}$ on $(\cJ, \cL)$} if, at every $p \in \cJ$, $g$ vanishes to order $(\beta_{p,\ell})_{\ell: \ell \ni p}$ in the directions $(\ell: \ell \ni p)$.

\begin{lemma} \label{lem:vanishing}
Let $(\cJ, \cL)$ be a joints configuration in $\FF^d$. Let $n$ be a nonnegative integer. 
Let $\alpha_p \in \ZZ$ for each $p \in \cJ$ and $\beta_{p, \ell} \in \ZZ_{\ge 0}$ for each $(p,\ell)\in \cJ \times \cL$ with $p \in \ell$ satisfying
\begin{enumerate}
	\item[(a)] $\beta_{p,\ell} - \alpha_p \geq \beta_{p',\ell} - \alpha_{p'}$ whenever $p,p' \in \ell$ and $\beta_{p',\ell}>0$, and 
	\item[(b)] $\sum_{p : p \in \ell} \beta_{p,\ell} \ge n$ for every $\ell \in \cL$.
\end{enumerate}
Then every nonzero polynomial that vanishes to order $\bm \beta = (\beta_{p,\ell})_{(p,\ell) : p \in \ell}$ on $(\cJ, \cL)$ has degree at least $n$.
\end{lemma}

\begin{remark}
It is fine to replace (a) by the simpler hypothesis that $\beta_{p,\ell} - \alpha_p = \beta_{p',\ell} - \alpha_{p'}$ whenever $p,p'\in \ell$ for proving \cref{thm:main} in this section. This special case of \cref{lem:vanishing} is slightly easier to think about though the proof remains essentially the same. We will need the above stated formulation in \cref{sec:higher-dimension}.
\end{remark}

\begin{proof}
Let $g$ be a nonzero polynomial that vanishes to order $\bm \beta$ on $(\cJ, \cL)$. Among all choices of $p \in \cJ$ and $(\gamma_1, \dots, \gamma_d) \in \ZZ_{> 0}^d$ such that $g$ does not vanish to order $(\gamma_1, \dots, \gamma_d)$ at $p$ along $(\ell_1,\dots, \ell_d)$, the $d$ distinct lines of $\cL$ passing through $p$, pick one so that $\gamma_1 + \cdots + \gamma_d - \alpha_p$ is minimized. By re-ordering these $d$ lines if necessary, we have $\gamma_1 > \beta_{p,\ell_1}$.
By a change of coordinates, let us assume that $p$ is the origin and $\ell_1, \dots, \ell_d$ are the coordinate axes. Write
\[
g(x_1, \dots, x_d) = \sum_{(\omega_2, \dots, \omega_d) \in \ZZ_{\ge 0}^{d-1}} g_{\omega_2, \dots, \omega_d}(x_1) x_2^{\omega_2} \cdots x_d^{\omega_d}.
\]
The coefficient of $x_1^{\gamma_1-1} \cdots x_d^{\gamma_d-1}$ in $g$ is nonzero since $g$ does not vanish to order $(\gamma_1, \dots, \gamma_d)$ at the origin along coordinate directions and $\gamma_1 + \cdots + \gamma_d$ is minimized for this chosen $p$. Thus $g_{\gamma_2 - 1,\dots, \gamma_d-1}$ is nonzero.

Let $p' = (c, 0, \dots, 0) \in \cJ$ be a joint on $\ell_1$, and $\ell'_2, \dots, \ell'_d$ the other lines in $\cL$ passing through $p'$.
We claim that $g_{\gamma_2 - 1,\dots, \gamma_d-1}(x_1)$ vanishes to order at least $\beta_{p', \ell_1}$ at $x_1 = c$. 
Assume $\beta_{p',\ell_1} > 0$ or else the claim vacuously holds.
For every $(\gamma'_1,\gamma'_2, \dots, \gamma'_d) \in \ZZ_{\ge 0}^{d-1}$ with $\gamma'_1+\gamma'_2 + \cdots + \gamma'_d = \beta_{p',\ell_1}+\gamma_2 + \cdots + \gamma_d$, applying hypothesis (a), we have
\begin{align*}
\gamma'_1 + \gamma'_2 + \cdots + \gamma'_d - \alpha_{p'} 
&= \beta_{p',\ell_1} + \gamma_2 + \cdots + \gamma_d - \alpha_{p'}
\\
&\le \beta_{p,\ell_1} + \gamma_2 + \cdots + \gamma_d - \alpha_{p}
\\
&< \gamma_1 + \gamma_2 + \cdots + \gamma_d - \alpha_{p},
\end{align*}
and thus, by minimality, $g$ vanishes to order $(\gamma'_1, \gamma'_2, \dots, \gamma'_d)$ at $p'$ along $(\ell_1, \ell'_2, \dots, \ell'_d)$. 

The above paragraph shows that for every joint $p'$ on $\ell_1$ formed by $\ell_1,\ell'_2,\ldots,\ell'_d$, if $\omega_1+\cdots+\omega_d$ is at most $(\beta_{p',\ell_1}-1)+(\gamma_2-1)+\cdots+(\gamma_d-1)$, then
\[[x_1^{\omega_1}\cdots x_d^{\omega_d}]g(p'+x_1e_{\ell_1}+x_2e_{\ell'_2}+\cdots+x_de_{\ell'_d}) = 0.\]
Since $e_{\ell_1},e_{\ell'_2},\ldots,e_{\ell'_d}$ form a basis, the same statement also holds when $e_{\ell'_2},\ldots,e_{\ell'_d}$ are replaced with $e_{\ell_2},\ldots,e_{\ell_d}$. By taking $\omega_i=\gamma_i-1$ for $i=2,\ldots,d$ and noting that $p'=p+ce_{\ell_1}$, we see that $g_{\gamma_2 - 1,\dots, \gamma_d-1}(x_1)$ vanishes to order at least $\beta_{p', \ell_1}$ at $x_1 = c$ for every joint $p' = (c, 0, \dots, 0)$ on $\ell_1$. Hence
$\deg g \ge \deg g_{\gamma_2-1, \dots, \gamma_d-1} \ge \sum_{p':p'\in \ell_1} \beta_{p',\ell} \ge n$
by (b).
\end{proof}

\begin{lemma}\label{lem:dis-ineq}
Assuming the same setup as \cref{lem:vanishing}, one has
\begin{equation} \label{eq:dis-ineq}
\sum_{p \in \cJ }\prod_{\ell \ni p} \beta_{p,\ell} \geq \binom{n+d-1}{d}.
\end{equation}
\end{lemma}

\begin{proof}
The set of all polynomials in $d$ variables with degree less than $n$ is a vector space with dimension $\binom{n+d-1}{d}$. 
The constraint that a polynomial vanishes to order $\bm \beta = (\beta_{p,\ell})_{(p,\ell) : p \in \ell}$ on $(\cJ, \cL)$ is a set of $\sum_p \prod_{\ell \ni p} \beta_{p,\ell}$ linear constraints. Thus, if \cref{eq:dis-ineq} is violated, then there exists some nonzero polynomial of degree less than $n$ that vanishes to order $\bm \beta$ on this joints configuration, which contradicts \cref{lem:vanishing}.
\end{proof}

\begin{remark}
One can show that \cref{ex:tight} is the exact optimum under the additional hypothesis that every line in $\cL$ contains exactly the same number $n$ of joints. Indeed, setting $\beta_{p,\ell}=1$ for all $(p,\ell)$ and $\alpha_p=0$ for all joints $p$ in \cref{lem:dis-ineq} yields $|\cJ| \ge \binom{n+d-1}{d} = \binom{d |\cJ|/ |\cL| + d-1}{d}$; in particular, $ \binom{k}{d-1}$ lines form at most $\binom{k}{d}$ joints. This observation appears to not have been stated before, though it can be derived using earlier techniques as it only considers single order vanishing.
\end{remark}

Earlier proofs~\cite{GK10,KSS10,Qui09} of the joints theorem first reduce to the case where every line can be assumed to contain many joints and then apply parameter counting to deduce the existence of a polynomial that vanishes to single order at every joint.
Zhang~\cite{Zhang} considers higher order vanishings to prove the multijoints extension.
Our control of higher order vanishing in \cref{lem:vanishing} is executed differently from earlier proofs. 
Furthermore, we set up a variational problem, below, associated to the problem of choosing the orders of vanishing.

\begin{lemma}\label{lem:cont-ineq}
Let $(\cJ, \cL)$ be a joints configuration in $\FF^d$. 
Suppose $a_p \in \RR$ for each $p \in \cJ$ and $b_{p, \ell} \in \RR_{\ge 0}$ for each $(p,\ell)\in \cJ \times \cL$ with $p \in \ell$ satisfy
\begin{enumerate}
	\item[(a)] $b_{p,\ell} - a_p = b_{p',\ell} - a_{p'}$ whenever $p,p' \in \ell$ and $b_{p',\ell}>0$, and 
	\item[(b)] $\sum_{p : p \in \ell} b_{p,\ell} = 1$ for every $\ell \in \cL$.
\end{enumerate}
Then 
\[
\sum_{p\in \cJ} \prod_{\ell \ni p} b_{p,\ell} \geq \frac{1}{d!}.
\]
\end{lemma}

\begin{proof}
Let $n$ be a sufficiently large positive integer (all the constants in this proof may depend on the joints configuration). Let $\alpha_p = \ceil{a_p n}$ for every $p \in \cJ$. 
Then we can find nonnegative integers $\beta_{p,\ell} = b_{p,\ell}n + O(1)$ so that $\beta_{p,\ell} - \alpha_p = \beta_{p',\ell} - \alpha_{p'}$ for every $p,p'\in\ell$ and $\ell\in\cL$, and furthermore $\sum_{p: p\in \ell} \beta_{p,\ell} \ge n$ for every $\ell \in \cL$. Thus by \cref{lem:dis-ineq}, we have
\[
\sum_p \prod_{\ell \ni p} (b_pn + O(1)) \ge \binom{n +d -1}{d}.
\]
Taking $n \to \infty$ yields the desired inequality.\end{proof}

We say that a joints configuration $(\cJ, \cL)$ is \emph{connected} if the graph constructed by taking $\cJ$ as vertices, with $p,p' \in \cJ$ adjacent if there is some $\ell \in \cL$ containing both $p$ and $p'$, is connected.

\begin{lemma}\label{lem:same-weight}
Let $(\cJ, \cL)$ be a connected joints configuration in $\FF^d$. 
Then there exist real numbers $a_p, b_{p,\ell}$ satisfying the hypothesis of \cref{lem:cont-ineq} with an additional constraint that $\prod_{\ell : \ell \ni p}b_{p,\ell}$ has the same value for all $p\in \cJ$. 
\end{lemma}

\begin{proof}
Let $B \subseteq [0,1]^{d|\cJ|}$ denote the set of feasible $\bm b = (b_{p,\ell}) \in [0,1]^{d|\cJ|}$ each satisfying the constraints of \cref{lem:cont-ineq} for some $(a_p)_{p\in \cJ}$. The set $B$ is closed since the existence of $(a_p) \in \RR^{|\cJ|}$ satisfying (a) is equivalent to a finite set of linear constraints on $\bm b$. Also $B$ is nonempty since setting all $b_{p,\ell} = |\ell \cap \cJ|^{-1}$ and $a_{p,\ell} = 0$ is feasible. 

By compactness, the quantity 
\begin{align}\label{eq:max-weight}
\max_{p\in \cJ}\prod_{\ell \ni p}b_{p,\ell}
\end{align}
is minimized by some $\bm b \in B$. For any $\bm b\in B$, call $p \in \cJ$ a \emph{maximizing joint} if it attains the maximum in \cref{eq:max-weight}. 

We claim that if some $\bm b$ minimizes the quantity in \cref{eq:max-weight}, then every joint is a maximizing joint, thus showing that $\prod_{\ell : \ell \ni p}b_{p,\ell}$ is the same for all $p \in \cJ$. 
Suppose for the sake of contradiction that this is not the case.
Among all possible global minimizers $\bm b$ for \cref{eq:max-weight}, choose one where the number of maximizing joints is minimum. Let $(a_p) \in \RR^{|\cJ|}$ be the parameters in \cref{lem:cont-ineq} associated to this $\bm b$.
Let us decrease $a_p$ at every maximizing joint $p$ by some sufficiently small $\epsilon > 0$. 
For every $\ell \in \cL$, if it contains all maximizing or all non-maximizing joints, then we do not need to change the values of $b_{p,\ell}$. 
For $\ell\in\cL$ that contains both maximizing and non-maximizing joints, we can strictly decrease $b_{p,\ell}$ at each maximizing $p \in \ell$ and increase each $b_{p,\ell}$ at each non-maximizing $p \in \ell$ so that the hypotheses (a) and (b) of \cref{lem:cont-ineq} remain satisfied for $\ell$. 
To see that this is always possible for sufficiently small $\epsilon$, we need to show that $b_{p,\ell}\neq 0$ when $p$ is maximizing and $b_{p,\ell}\neq 1$ when $p$ is non-maximizing. 
If $b_{p,\ell}=0$ for some maximizing joint $p$ on the line $\ell$, then for any joint $p'$ the product $\prod_{\ell'\ni p'}b_{p',\ell'}$ has to be zero, implying that all joints are maximizing, which is a contradiction. 
In addition, if $b_{p,\ell}=1$ for some non-maximizing joint $p$ on the line $\ell$, then by assumption there is another maximizing joint $p'$ on the same line with $b_{p',\ell}=0$ by the hypothesis (b) of \cref{lem:cont-ineq}, which is also a contradiction.

After the modification described in the previous pargraph, the product $\prod_{\ell : \ell \ni p} b_{p,\ell}$ never increases at maximizing joints $p$, and is in fact strictly lowered whenever $p$ shares a line with some non-maximizing joint. 
Due to the connectivity of the joints configuration and the assumption that not all joints are maximizing, the value of $\prod_{\ell : \ell \ni p} b_{p,\ell}$ must be strictly lowered at some maximizing joint. 
On one hand, if $\prod_{\ell:\ell\ni p}b_{p,\ell}$ stays the same for some maximizing joint $p$, then this contradicts the minimality of number of maximizing joints. 
On the other hand, if $\prod_{\ell:\ell\ni p}b_{p,\ell}$ decreases for all maximizing joints $p$, then \cref{eq:max-weight} also decreases, contradicting the assumption that $\bm b$ is a minimizer for \cref{eq:max-weight}.
Therefore, all joints are maximizing.
\end{proof}

\begin{proof}[Proof of \cref{thm:main}]
Let $(\cJ, \cL)$ be a joints configuration with $|\cJ| = J$ and $|\cL| = L$.
First suppose that the joints configuration is connected. Choose $a_p, b_{p,\ell}$ as in \cref{lem:same-weight}. Let $W$ denote the common value of  $\prod_{\ell : \ell \ni p}b_{p,\ell}$. By the AM-GM inequality followed by hypothesis (b) in \cref{lem:cont-ineq},
\[
d J W^{1/d}
= \sum_p d  \biggl( \prod_{\ell \ni p} b_{p,\ell} \biggr)^{1/d}
\le \sum_{p} \sum_{\ell \ni p} b_{p,\ell}
= \sum_\ell  \sum_{p \in \ell} b_{p,\ell}
= L,
\]
and thus
\[
W \le \frac{L^d}{d^dJ^d}.
\]
Thus, by \cref{lem:cont-ineq},
\[
\frac{1}{d!}\leq \sum_p\prod_{\ell \ni p}b_{p,\ell } 
= JW 
\leq \frac{L^d}{d^d J^{d-1}}.
\]
Thus $J \le C_d L^{d/(d-1)}$ with $C_d = (d-1)!^{1/(d-1)}/d$.

Finally, decompose $(\cJ, \cL)$ into connected components $(\cJ_1, \cL_1), \dots, (\cJ_k, \cL_k)$, i.e., connected components of the associated graph in the definition of connectivity, and apply the above result individually to each component to obtain
\[
|\cJ| = \sum_{i=1}^{k} |\cJ_i|
\leq C_d  \sum_{i=1}^{k} |\cL_i|^{d/(d-1)} \leq C_d |\cL|^{d/(d-1)}. \qedhere
\]
\end{proof}

\section{Multijoints}\label{sec:multijoint}

We say that $(\cJ, \cL_1, \dots, \cL_d)$ is a \emph{multijoints configuration} in $\FF^d$ if each $\cL_i$ is a set of lines and $\cJ$ is a set of points each being the intersection of exactly one line from each $\cL_i$, not all lying on the same hyperplane. 
The points in $\cJ$ are called the \emph{multijoints}.

The following generalization of the joints theorem was conjectured by Carbery and proved in $\FF^3$ and $\RR^n$ by Iliopoulou~\cite{Ili15a,Ili15b} and in general $\FF^n$ by Zhang~\cite{Zhang}.
Zhang also proves a further generalization that counts with multiplicities when a joint is  contained in many lines, although we do not discuss it here.

\begin{theorem}[Multijoints~\cite{Ili15a,Ili15b,Zhang}] \label{thm:zhang-multijoint}
	For every $d$ there is some constant $C^{\mathrm{mult}}_d$ so every multijoints configuration $(\cJ, \cL_1, \dots, \cL_d)$ in $\FF^d$ satisfies $|\cJ| \le C^{\mathrm{mult}}_d (\abs{\mathcal L_1}\cdots \abs{\mathcal L_d})^{1/(d-1)}$.
\end{theorem}

\begin{remark}
    \cref{thm:zhang-multijoint} implies \cref{thm:joints}, and furthermore they are equivalent if all $|\cL_i|$ are within a constant factor of each other. However, the  constant factors are lossy under the reductions.
\end{remark}

Zhang's proof gives \cref{thm:zhang-multijoint} with $C^{\mathrm{mult}}_d = d^{d/(d-1)}$. We improve this constant factor using our method, though our bound is also likely not optimal.

\begin{theorem} \label{thm:multijoints-improvement}
	\cref{thm:zhang-multijoint} holds with $C^{\mathrm{mult}}_d = d!^{1/(d-1)}$.
\end{theorem}

\begin{proof}
The proof is the same as that of \cref{thm:main} except for the final calculation. As before, we can reduce to case of a connected configuration. Choose $a_p, b_{p,\ell}$ as in \cref{lem:same-weight}. Write $L_i = |\cL_i|$, and $b_{p,i} = b_{p,\ell}$ where $\ell$ is the line in $\cL_i$ that contains $p$. Let $W$ denote the common value of  $\prod_{i=1}^d b_{p,i}$. 
By the AM-GM inequality,
\[
\frac{d J W^{1/d}}{\left(L_1 \cdots L_d\right)^{1/d}}
= \sum_p d  \biggl( \prod_{i=1}^d \frac{b_{p,i }}{L_i} \biggr)^{1/d}
\le \sum_{p} \sum_{i=1}^d \frac{b_{p,i}}{L_i} = \sum_{i=1}^d \frac{1}{L_i} \sum_{\ell \in \cL_i} \sum_{p \in \ell} b_{p,i} 
= d.
\]
So 
\[
W\leq \frac{|L_1|\cdots|L_d|}{|J|^d},
\]
and hence by \cref{lem:cont-ineq},
\[
\frac{1}{d!} 
\le \sum_p \prod_{i=1}^d b_{p,i}
= J W
\leq \frac{|L_1|\cdots|L_d|}{|J|^{d-1}}.
\]
Hence
\[
J \le \paren{ d! |L_1|\cdots|L_d|}^{1/(d-1)}. \qedhere 
\]
\end{proof}

We propose a conjecture for the optimal configuration for multijoints.
We say that a joints or multijoints configuration is \emph{generically induced} if all the lines come from taking $(d-1)$-wise intersections of generic hyperplanes and all the joints come from taking $d$-wise intersections of the same hyperplanes (though we do not have to include all possible intersections). Here, a set of hyperlanes is generic if any $d$-wise intersection of the hyperplanes is a point, and any $(d+1)$-wise intersection is empty. In other words, we consider a subcollection of the construction in \cref{ex:tight}. We conjecture that the optimal configurations for multijoints are generically induced.

A generically induced multijoints configuration in $\FF^3$ can be encoded by an edge-colored graph, with each vertex corresponding to one of the generic planes.
We color the edges of a complete graph using three colors, where every edge is allowed to receive any number (including zero) of colors. 
For each $i=1,2,3$, assign to an edge $uv$ color $i$ if $\cL_i$ contains the line formed by the intersections of the two hyperplanes corresponding to $u$ and $v$. 
Then $|\cL_i|$ is the number of edges with color $i$. The number of multijoints is the number of rainbow triangles, i.e., triangles formed by taking one edge of each color (each triangle is counted at most once even if multiple rainbow color combinations are available). The process is reversible: starting with such an edge-colored graph, we can recover a generically induced multijoints configuration.

Thus, when restricted to generically induced multijoints configurations in $\FF^3$, the multijoints problem is purely graph theoretic: given $L_1, L_2, L_3$, what is the maximum number of rainbow triangles if we color $L_1$ edges red, $L_2$ edges blue, and $L_3$ edges yellow (again, every edge is allowed to receive any number of colors, including none)?

Here is an example.
Consider the edge-coloring of the clique $K_4$ by decomposing its edge-set into three matchings and assigning a different color to each matching. By blowing up each vertex of $K_4$ into $k$ vertices, we obtain an edge-coloring of a complete 4-partite graph with $L = 2k^2$ edges of each color and $J = 4k^3$ rainbow triangles. We believe that this construction gives the optimal constant factor for the extremal graph theoretic problem stated in the previous paragraph, i.e., there are always at most $\sqrt{2L_1L_2L_3}$ rainbow triangles (we thank Bernard Lidicky for running a preliminary flag algebra computation for us in the case of equal number of edges of each color). A more difficult conjecture is that this geometric configuration is also optimal for multijoints.

\begin{conjecture}
    \cref{thm:zhang-multijoint} for $d = 3$ holds with $C^{\mathrm{mult}}_3 = \sqrt{2}$.
\end{conjecture}

In higher dimensions, generically induced configurations correspond to rainbow simplices in edge-colored hypergraphs, and it seems an interesting combinatorial problem to determine the optimal constant even in this special case.

The joints problem with one set of lines (i.e., the setting of \cref{thm:joints,thm:main}) can also be considered in this light.
For a generically induced joints configuration in $\FF^d$, the joints problem is equivalent to determining the maximum number of simplices in a $(d-1)$-uniform hypergraph with a given number $L$ of edges, and this problem is completely understood due to the classic Kruskal--Katona theorem.
The Kruskal--Katona theorem also suggests the answer for the exact maximum number of joints when the number $L$ of lines is not a number of the form $\binom{k}{d-1}$.
Curiously, our proof of \cref{thm:main} seems to give a new proof an asymptotic version of this special case of Kruskal--Katona theorem (i.e., up to $1+o(1)$ factor) via the polynomial method, namely that a $(d-1)$-uniform hypergraph with $L$ edges has at most $C_d L^{d/(d-1)}$ simplices, where $C_d = (d-1)!^{1/(d-1)}/d$.

\section{Joints of flats}\label{sec:higher-dimension}

Consider the following higher dimensional generalization of the joints problem. A \emph{$k$-flat} is a $k$-dimensional affine subsapce.
Given a collection of 2-flats in $\RR^6$, we say that a point is a \emph{joint} if it is lies on a triple of the given 2-flats not all contained in some hyperplane. What is the maximum number of joints formed by $N$ 2-flats?

A construction analogous to \cref{ex:tight} works for the higher dimensional setting. Alternatively, by considering a joints configuration in $\CC^3$ and then viewing the complex lines them as real 2-flats in $\RR^6$, one obtains $\Theta(N^{3/2})$ joints. It remains a very interesting open problem to prove an $O(N^{3/2})$ upper bound.

The best known upper bound on the number of joints of 2-flats in $\RR^6$ is $N^{3/2+o(1)}$ due to Yang~\cite{Yang}, who proved the result using a polynomial partitioning technique (and hence his method only works in $\RR$ and not in arbitrary fields). Yang also proves a similar claim about higher dimensional varieties of bounded degree instead of 2-flats, as well as a ``multijoints'' generalization. 

Following Yang's result, we consider the following generalization of the joints problem. Let $d_1, \dots, d_r$ and $m_1, \dots, m_r$ be positive integers. 
We say that $(\cJ, \cF_1, \dots, \cF_k)$ is a \emph{$(d_1, \dots, d_r; m_1, \dots, m_r)$-joints configuration}  
if each $\cF_i$ is a set of $d_i$-flats in $\FF^d$, where $d = d_1m_1 + \cdots + d_rm_r$, 
and $\cJ$ is a set of joints each being the intersection of $m_i$ elements from each $\cF_i$ for $i \in [r]$, not all lying in some hyperplane.

For example, the joints configuration of \cref{thm:joints} corresponds to the parameter $(1;d)$. The multijoints configuration of \cref{thm:zhang-multijoint} corresponds to $(1, \dots, 1;1,\dots,1)$. The joints problem of 2-flats in $\RR^6$ corresponds to $(2;3)$.

Yang~\cite{Yang} proved the bound $|\cJ| \le (|\cF_1|^{m_1} \cdots |\cF_r|^{m_r})^{1/(m_1 + \cdots + m_r-1) +  o(1)}$ when $\FF = \RR$. It is natural to conjecture the following bound.

\begin{conjecture} \label{conj:higher}
    Every $(d_1, \dots, d_r; m_1, \dots, m_r)$-joints configuration $(\cJ, \cF_1, \dots, \cF_k)$ in $\FF^d$ satisfies
	\[
	|\cJ| \le C_{d_1, \dots, d_r; m_1, \dots, m_r} (|\cF_1|^{m_1} \cdots |\cF_r|^{m_r})^{1/(m_1 + \cdots + m_r-1)}.
	\]
	for some constant $C_{d_1, \dots, d_r; m_1, \dots, m_r}$.
\end{conjecture}

(Note: The above conjecture is proved by Tidor, Yu, and Zhao in a subsequent work~\cite{TYZ}.)

Furthermore, we conjecture that the optimal constant comes from generically induced configurations formed by placing generic hyperplanes so that the flats are the intersections of an appropriate number of these hyperplanes. As in \cref{sec:multijoint}, this special case can be encoded as a face-colored simplicial complex (non-uniform hypergraph), where joints correspond to simplices that are colored in a certain way. 

As in the remark after \cref{thm:zhang-multijoint}, \cref{conj:higher} for $(d_1, d_2, \dots, d_r; m_1, m_2, \dots, m_r)$ is implied by the case $(d_1, \dots, d_1, d_2, \dots, d_2, \dots, d_r, \dots, d_r;1,\dots,1)$ where each $d_i$ is repeated $m_i$ times. However, we state it in the above form since it suggests a hierarchy of difficulties for the conjecture, and also since the optimal constants $C_{\bm d}$ are not preserved under the reduction.

Here we record a proof of a special case of  \cref{conj:higher} (with a likely non-optimal constant) by a variation of our techniques.

\begin{theorem}\label{thm:higher}
    Every $(1, d-m;m,1)$-joints configuration $(\cJ, \cL, \cF)$ in $\FF^d$ satisfies $|\cJ| \le \binom{d}{m}^{1/m} |\cL| |\cF|^{1/m}$	
\end{theorem}

Throughout, $\cL$ is a set of lines and $\cF$ is a set of $(d-m)$-flats in $\FF^d$. 
We similarly abuse the notation as in \cref{sec:proof} when we say that a joint is ``contained'' in a line or a flat .

For a polynomial $g\in \FF[x_1,\ldots,x_d]$ and a $(d-m)$-flat $f$, we say that $g$ \emph{vanishes to order $\gamma$ on $f$} if, after an affine transformation taking $f$ to $x_1 = \cdots = x_m = 0$, the coefficient of $x_1^{\omega_1}\cdots x_d^{\omega_d}$ in $g$ is zero whenever $\omega_1+\cdots+\omega_m < \gamma$. Moreover, we say that $g$ \emph{vanishes to order $\bm\gamma=(\gamma_f)_{f\in \cF}\in \ZZ_{> 0}^{|\cF|}$ on $\cF$} if $g$ vanishes to order $\gamma_f$ on $f$ for all $f\in \cF$.

\begin{lemma}\label{lem:higher-vanishing}
Let $(\cJ, \cL, \cF)$ be a $(1,d-m;m,1)$-joints configuration in $\FF^d$. Let $n$ be a nonnegative integer. Let $\alpha_p \in \ZZ$ for each $p \in \cJ$, $\beta_{p, \ell} \in \ZZ_{\ge 0}$ for each $(p,\ell)\in \cJ \times \cL$ with $p \in \ell$, and $\gamma_{f}\in \ZZ_{\ge 0}$ for each $f\in \cF$ satisfying
\begin{enumerate}
	\item[(a)] $\beta_{p,\ell} - \alpha_p \ge \beta_{p',\ell} - \alpha_{p'}$ whenenver $p,p' \in \ell$ and $\beta_{p',\ell} > 0$,
	\item[(b)] $\sum_{p : p \in \ell} \beta_{p,\ell} \ge n$ for every $\ell \in \cL$, and
	\item[(c)] $\sum_{\ell: \ell \ni p}\beta_{p,\ell}\leq \gamma_f$ for every $p\in \cJ, f\in \cF$ such that $p\in f$.
\end{enumerate}
Then every nonzero polynomial that vanishes to order $\bm \gamma = (\gamma_f)_{f\in \cF}$ on $\cF$ has degree at least $n$.
\end{lemma}
\begin{proof}
Let $\cL'$ be the union of $\cL$ along with, for each pair $p \in f$, $d-m$ new lines in $f$ all passing through $p$ in linearly independent directions. 
Then $(\cJ, \cL')$ is a joints configuration, where each new line contains exactly one joint. 
Set $\beta_{p,\ell} = n$ for these new lines. 
Then $(\alpha_p)_{p\in \cJ}$ and $\bm \beta' = (\beta_{p,\ell})_{(p,\ell) \in \cJ \times \cL':p\in \ell}$ satisfy the setup of \cref{lem:vanishing}. For every nonzero polynomial $g$ that vanishes to order $\bm \gamma$ on $\cF$, consider any joint $p$. By taking a linear transformation, without loss of generality assume that $p$ is a joint formed by the $x_i$-axes $(i=1,\ldots,d)$, denoted as $\ell_i$, where $\ell_i$ $(i=m+1,\ldots,d)$ are the lines added in $f$ for $p$. Then
\[[x_1^{\omega_1}\cdots x_d^{\omega_d}]g(x_1,\ldots,x_d)=0\]
whenever $\omega_1+\cdots+\omega_m<\gamma_f$. In particular, this holds whenever $\omega_i<\beta_{p,\ell_i}$ for $i=1,\ldots,d$ by assumption (c). Therefore, $g$ also vanishes to order $\bm\beta'$ on $(\cJ, \cL')$ by assumption (c).
By \cref{lem:vanishing}, the degree of $g$ is at least $n$.
\end{proof}

\begin{lemma}\label{lem:higher-dis-ineq}
Assuming the same setup as \cref{lem:higher-vanishing}, one has
\begin{equation} \label{eq:higher-dis-ineq}
\binom{n+d-m-1}{d-m}\sum_{f\in \cF} \binom{\gamma_f+m-1}{m} \geq \binom{n+d-1}{d}.
\end{equation}
\end{lemma}
\begin{proof}
The vector space of polynomials in $d$ variables with degree less than $n$ has dimension $\binom{n+d-1}{d}$. The constraint that a polynomial vanishes to order $\bm \gamma$ on $\cF$ is a set of at most $\binom{n+d-m-1}{d-m}\sum_{f\in \cF} \binom{\gamma_f+m-1}{m}$ linear constraints. Therefore, if \cref{eq:higher-dis-ineq} does not hold, then there exists a nonzero polynomial that vanishes to order $\bm\gamma$ but has degree less than $n$, contradicting \cref{lem:higher-vanishing}.
\end{proof}

\begin{lemma}\label{lem:higher-cont-ineq}
Let $(\cJ, \cL, \cF)$ be a $(1,d-m;m,1)$-joints configuration in $\FF^d$. Let $a_p \in \RR$ for each $p \in \cJ$, $b_{p, \ell} \in \RR_{\ge 0}$ for each $(p,\ell)\in \cJ \times \cL$ with $p \in \ell$, and $c_{f}\in \RR_{\ge 0}$ for each $f\in \cF$ satisfying
\begin{enumerate}
	\item[(a)] $b_{p,\ell} - a_p \ge b_{p',\ell} - a_{p'}$ whenever $p,p' \in \ell$ and $b_{p',\ell} > 0$,
	\item[(b)] $\sum_{p : p \in \ell} b_{p,\ell} = 1$ for every $\ell \in \cL$, and
	\item[(c)] $\sum_{\ell: \ell \ni p}b_{p,\ell}\leq c_f$ for every $p\in \cJ, f\in \cF$ such that $p\in f$.
\end{enumerate}
Then
\[\sum_{f\in \cF}c_f^{m}\geq \binom{d}{m}^{-1}.\]
\end{lemma}
\begin{proof}
Let $n$ be a positive integer. Let $\alpha_p=\lceil a_pn\rceil$ for every $p\in \cJ$. By the assumptions we can take $\beta_{p,\ell}=b_{p,\ell}n+O(1)$ and $\gamma_f = c_f+O(1)$ so that the setup of \cref{lem:higher-dis-ineq} is met. By \cref{lem:higher-dis-ineq},
\[\binom{n+d-m-1}{d-m}\sum_{f\in \cF}\binom{c_fn+O(1)}{m}\geq \binom{n+d-1}{d},\]
and by taking $n$ to infinity, we obtain
\[\frac{1}{(d-m)!}\sum_{f\in \cF}\frac{c_f^{m}}{m!}\geq\frac{1}{d!},\]
which is the claimed inequality.
\end{proof}
\begin{lemma}\label{lem:same-sum}
For every $(1,d-m;m,1)$-joints configuration $(\cJ, \cL, \cF)$, there exist nonempty subsets $\cJ'\subseteq \cJ$ and  $\cL'\subseteq \cL$ such that $(\cJ', \cL', \cF)$ is also a $(1,d-m;m,1)$-joints configuration with $(a_p)_{p\in \cJ'}$,  $(b_{p,\ell})_{(p,\ell) \in \cJ' \times \cL' :p\in \ell}$, $(c_f)_{f\in \cF}$ that satisfy the setup of \cref{lem:higher-cont-ineq} for $(\cJ', \cL', \cF)$, and furthermore,
\begin{enumerate}
    \item[(a)] $|\cJ'|/ |\cL'| \ge |\cJ|/|\cL|$, and
    \item[(b)] there is some $s$ such that $\sum_{\ell:\ell\ni p}b_{p,\ell} = s$ and $c_f = s$ for every $p\in \cJ'$ and $f\in \cF'$.
\end{enumerate}
\end{lemma}

\begin{proof}
Let us replace $\cJ$ and $\cL$ by their nonempty subsets $\cJ'$ and $\cL'$ such that $(\cJ', \cL', \cF)$ is a $(1, d-m;m,1)$-joints configuration with $|\cJ'|/|\cL'|$ maximized. It remains to show that (b) can be satisfied on this new $(\cJ, \cL, \cF)$.

Let $B\subseteq [0,1]^{m|\cJ|}$ be the set of feasible $\bm b=(b_{p,\ell})\in [0,1]^{m|\cJ|}$ satisfying the constraints of \cref{lem:higher-cont-ineq}. Then by the compactness and non-emptiness of $B$ again, the quantity
\begin{equation}\label{eq:max-sum}
\max_{p\in \cJ}\sum_{\ell \ni p}b_{p,\ell}
\end{equation}
is minimized by some $\bm b\in B$. We still call $p\in \cJ$ a \emph{maximizing joint} if it attains the maximum in \cref{eq:max-sum}. Once again, we will show that if $\bm b$ is a minimizer for \cref{eq:max-sum}, then all joints are maximizing. If this is the case, then the statement follows by setting $c_f$ to be the quantity in \cref{eq:max-sum}.

Assume for the sake of contradiction that not all joints are maximizing for some minimizer $\bm b$. Among all the minimizers $\bm b$ for \cref{eq:max-sum}, take one where the number of maximizing joints is minimum. Let $(a_p)_{p\in\cJ}\in \RR^{|\cJ|}$ be the parameters in \cref{lem:higher-cont-ineq} corresponding to this $\bm b$.

We claim that if $\ell \in \cL$ contains a non-maximizing joint, then every maximizing joint $p \in \ell$ has $b_{p,\ell} = 0$. Indeed, if not, then we can decrease $a_p$ at every maximizing joint $p \in \cJ$ by some sufficiently small $\epsilon > 0$, and for every $\ell\in \cL$ containing both maximizing and non-maximizing joints, we can strictly decrease any nonzero $b_{p,\ell}$ at each maximizing $p\in\ell$ and increase each $b_{p,\ell} < 1$ at each non-maximizing $p\in\ell$ accordingly so that the hypotheses (a) and (b) of \cref{lem:higher-cont-ineq} still hold. If $\ell$ contains a non-maximizing joint as well as a maximizing joint $p$ with $b_{p,\ell} > 0$, then this decrement contradicts the minimality of $\bm b$ or the number of maximizing joints.

Let $\cJ'$ be the set of non-maximizing joints and $\cL'$ be the set of lines in $\cL$ that contain non-maximizing joints. Then $\cJ'$ and $\cL'$ are nonempty by the assumption. Moreover, we saw that $b_{p,\ell} = 0$ whenever $(p,\ell) \in (\cJ \setminus \cJ') \times \cL'$. Also, there are no $(p,\ell) \in \cJ' \times (\cL\setminus \cL')$ with $p \in \ell$ by the definition of $\cL'$. Using $\sum_{p\in \ell}b_{p,\ell}=1$ for every $\ell \in \cL$, and denoting the quantity in \cref{eq:max-sum} by $s$, we have
\[
s|\cJ\setminus \cJ'| = \sum_{p \text{ maximizing}} \sum_{\ell \ni p} b_{p,\ell}
= \sum_{\substack{(p,\ell) \in (\cJ \setminus \cJ') \times (\cL \setminus \cL'):\\p\in\ell}} b_{p,\ell} 
= \sum_{\ell \in \cL \setminus \cL'} \sum_{p \in \ell} b_{p, \ell} 
= |\cL \setminus \cL'|\]
and
\[s|\cJ'| > \sum_{p\text{ non-maximizing}}\sum_{\ell\ni p}b_{p,\ell} = \sum_{\ell\in \cL'}\sum_{p\in \ell} b_{p,\ell} = |\cL'|.\]
So $|\cJ'|/|\cL'| > |\cJ|/|\cL|$, which contradicts the maximality assumption at the beginning of the proof. Hence the conclusion holds.
\end{proof}

\begin{proof}[Proof of \cref{thm:higher}]
Let $(\cJ, \cL, \cF)$ be a $(1,d-m;m,1)$-joints configuration with $|\cJ| = J, |\cL| = L$ and $|\cF|=F$. Choose $\cJ', \cL', a_p, b_{p,\ell}, c_f, s$ as in \cref{lem:same-sum}. Let $|\cJ'|=J'$ and $|\cL'|=L'$. Then
\[sJ' = \sum_{p\in \cJ'}\sum_{\ell\in \cL':\ell\ni p}b_{p,\ell} = \sum_{\ell\in \cL'}\sum_{p\in \ell \cap \cJ'}b_{p,\ell} = L'\]
and so
\[s = \frac{L'}{J'}\le \frac{L}{J}.\]
Now by \cref{lem:higher-cont-ineq} applied on $(\cJ', \cL', \cF)$,
\[\binom{d}{m}^{-1} \le \sum_{f\in \cF}c_f^{m} = Fs^{m}\le \frac{FL^{m}}{J^{m}}.\]
Thus
\[
J\le \binom{d}{m}^{1/m}LF^{1/m}. \qedhere 
\]
\end{proof}

\subsection*{Acknowledgments} 
YZ thanks Boris Bukh for hosting him during a visit at CMU in February 2018 and for discussions that eventually led to \cref{sec:higher-dimension}. 
We thank Ben Lund for sharing with us his independent results that improve the constant in \cref{thm:joints} to $C_d = 1$. 
After sharing our initial draft of the paper, we have learned that Tony Carbery and Marina Iliopoulou \cite{CI20} have independently proved Theorem 4.2 (without specifying an explicit constant) using a different technique.
We thank Adam Sheffer and Ruixiang Zhang for comments on the initial draft of the paper, and the anonymous referee for a careful reading and helpful comments.

Yu was supported by the MIT Undergraduate Research Opportunities Program. Zhao was supported by NSF Award DMS-1764176, the MIT Solomon Buchsbaum Fund, and a Sloan Research Fellowship.


\end{document}